\documentclass[11pt,reqno]{amsart}
\usepackage[T1]{fontenc}
\usepackage{lmodern}
\usepackage{amsmath,amssymb,amsthm,mathtools}
\usepackage[margin=1.12in]{geometry}
\usepackage[expansion=false]{microtype}
\usepackage{enumitem}
\usepackage{booktabs}
\usepackage{xcolor}
\usepackage[hypertexnames=false,colorlinks=true,linkcolor=blue!45!black,citecolor=blue!45!black,urlcolor=blue!45!black]{hyperref}
\numberwithin{equation}{section}
\newtheorem{theorem}{Theorem}[section]

\newtheorem{lemma}[theorem]{Lemma}

\theoremstyle{definition}
\newtheorem{definition}[theorem]{Definition}
\newtheorem{example}[theorem]{Example}
\theoremstyle{remark}
\newtheorem{remark}[theorem]{Remark}
\newcommand{\R}{\mathbb R}
\newcommand{\HH}{\mathcal H}
\newcommand{\Reg}{\mathcal R}
\newcommand{\Sing}{\mathcal S}
\newcommand{\Man}{\mathcal U}
\newcommand{\Ric}{\operatorname{Ric}}

\newcommand{\Vol}{\operatorname{Vol}}
\newcommand{\diam}{\operatorname{diam}}
\newcommand{\dist}{\operatorname{dist}}
\newcommand{\dimH}{\dim_{\mathrm H}}
\newcommand{\GH}{\mathrm{pGH}}
\newcommand{\Tube}{\mathcal T}

\setlist[enumerate]{label=\textup{(\roman*)},leftmargin=2.1em}
\title[Volume deficits and codimension-three regularity]{Codimension-three regularity of noncollapsed Ricci limit spaces under an integral volume-deficit bound}
\author{Lingling Kong}
\address[Lingling Kong]{School of Mathematical and statistical,  Northeast Normal University, Changchun, China}
\email{kongll111@nenu.edu.cn}

\date{\today}

\begin{document}
\begin{abstract}
Let $(X,d,p)$ be a noncollapsed pointed Gromov--Hausdorff limit of complete
$n$-dimensional Riemannian manifolds with a uniform lower Ricci curvature
bound, where $n\ge4$. We assume that, on each bounded ball, the integral
of the $3/2$ power of the small-ball volume deficit relative to the
hyperbolic comparison volume is $O(r^3)$ as $r\rightarrow0$. 
We prove that the metric singular set has Hausdorff dimension
at most $n-3$ and sigma-finite $(n-3)$-dimensional Hausdorff measure, thus confirming a particular case of  codimension-three regularity conjecture \cite[Conjecture 2.4]{Naber2020}. 
\end{abstract}
\maketitle

\section{Introduction and main theorem}

Throughout this paper, a noncollapsed Ricci limit space is a pointed
Gromov--Hausdorff limit
\begin{equation}\label{eq:setting}
 (M_i^n,g_i,p_i)\xrightarrow{\GH}(X,d,p),\qquad
 \Ric_{g_i}\ge -(n-1)g_i,\qquad
 \Vol_{g_i}(B_1(p_i))\ge v>0,
\end{equation}
where the manifolds are connected, complete, and without boundary. We
write $\mu=\HH^n$, with Hausdorff measure normalized to agree with
Riemannian volume, and let $\omega_n$ be the volume of the Euclidean unit
ball. Volume convergence identifies $\mu$ with the limit of the
Riemannian volume measures \cite{Colding1997,CC1997I}.

Cheeger--Colding theory provides a metric regular set $\Reg$, consisting
of points whose tangent cones are all isometric to $\mathbb{R}$, and a metric singular
set $\Sing=X\setminus\Reg$ of Hausdorff dimension at most $n-2$.
It has been conjectured  that every noncollapsed Ricci limit spaces is a topological manifold outside a closed set of Hausdorﬀ codimension at least 4: see \cite[ Conjecture 0.7]{CC1997I}, 
\cite[Remark 10.23]{Cheeger2001} and \cite[Remark 1.19]{CheegerColdingTian2002}.  
The codimension-three formulation
appears in \cite[Conjecture 2.4]{Naber2020}. In this paper, we establish the codimension-three conclusion under an
additional integral bound on small-ball volume deficits. Our estimate
applies to the metric singular set and therefore also yields the
corresponding topological conclusion.

Define the hyperbolic comparison volume and the volume deficit by
\begin{equation}\label{eq:VandD}
 V(r)=n\omega_n\int_0^r(\sinh t)^{n-1}\,dt,
 \qquad
 D_r(x)=1-\frac{\mu(B_r(x))}{V(r)}.
\end{equation}
Here and below $B_r(x)$ is an open metric ball. Bishop--Gromov comparison
implies $0\le D_r(x)\le1$, and $r\mapsto D_r(x)$ is nondecreasing.

\begin{definition}[Integral volume-deficit condition]\label{def:A}
We say that $X$ satisfies \textup{(A)} if, for every $R>0$, there are
constants $C_R<\infty$ and $r_R>0$ such that
\begin{equation}\tag{A}\label{eq:A}
 \int_{B_R(p)}D_r(x)^{3/2}\,d\mu(x)\le C_Rr^3
 \qquad (0<r<r_R).
\end{equation}
The constants may depend on the space and on $R$, but are independent of
the small radius $r$.
\end{definition}

Define the volume density and the associated density-deficit sets by
\begin{equation}\label{eq:density}
 \Theta(x)=\lim_{r\rightarrow0}\frac{\mu(B_r(x))}{V(r)},
 \qquad
 S_\delta=\{x\in X:\Theta(x)\le1-\delta\},\quad 0<\delta<1.
\end{equation}
For a subset $F\subset X$ and $\rho>0$, write
\[
 \Tube_\rho(F)=\{y\in X:\dist(y,F)<\rho\}.
\]
Let $N_\rho(F)$ denote the smallest number of open balls of radius $\rho$,
with centers in $X$, required to cover $F$.

\begin{theorem}\label{thm:main}
Suppose that \eqref{eq:setting} holds, $n\ge4$, and $X$ satisfies
\textup{(A)}. Then the following conclusions hold.
\begin{enumerate}
\item For every $R>0$ and $0<\delta<1$, there are constants
$K_{R,\delta}<\infty$ and $\rho_{R,\delta}>0$ such that
\begin{align}
 N_\rho(S_\delta\cap B_R(p))&\le K_{R,\delta}\rho^{3-n},
 \label{eq:main-cover}\\
 \mu\bigl(\Tube_\rho(S_\delta\cap B_R(p))\bigr)
 &\le K_{R,\delta}\rho^3
 \label{eq:main-tube}
\end{align}
for $0<\rho<\rho_{R,\delta}$. In particular,
\begin{equation}\label{eq:main-finite}
 \HH^{n-3}(S_\delta\cap B_R(p))<\infty.
\end{equation}
\item The metric singular set satisfies
\begin{equation}\label{eq:main-dim}
 \dimH\Sing\le n-3,
\end{equation}
and $\HH^{n-3}$ is sigma-finite on $\Sing$.
\item Define the topological manifold locus by
\begin{equation}\label{eq:manifold-locus}
 \Man=\{x\in X:\text{$x$ has an open neighborhood homeomorphic to $\R^n$}\}.
\end{equation}
Then $E=X\setminus\Man$ is closed,
$\dimH E\le n-3$, and $X\setminus E$ is a topological
$n$-manifold without boundary. Moreover, $\HH^{n-3}$ is sigma-finite on $E$.
\end{enumerate}
\end{theorem}

The additional hypothesis \textup{(A)} does not follow from the lower
Ricci curvature bound and noncollapsing alone, as the cone example in
Section~\ref{sec:examples} shows. Theorem~\ref{thm:main} therefore
establishes the codimension-three manifold conclusion under this
additional hypothesis. In fact, it gives the stronger conclusion that
the \emph{metric} singular set has Hausdorff codimension at least three.

\begin{remark}\label{rem:distinctions}
The sets $S_\delta$ are defined in terms of volume density and should be
distinguished from the strata defined by Euclidean splitting of tangent cones. The constants in
\eqref{eq:main-cover}--\eqref{eq:main-tube} may deteriorate as
$\delta\rightarrow0$. Thus the theorem does not yield a uniform tubular-volume
estimate or local finiteness of $\HH^{n-3}$ on the entire set $\Sing$.
The dimension bound \eqref{eq:main-dim} applies to $\Sing$ itself,
not necessarily to its closure.
\end{remark}

\section{Volume comparison and density deficits}\label{sec:prelim}

We recall the consequences of Ricci limit theory needed in the proof. The space $X$ is proper and separable, $\mu$ is finite on bounded
sets and has full support, and
\begin{equation}\label{eq:BG}
	\frac{\mu(B_s(x))}{V(s)}\ge
	\frac{\mu(B_t(x))}{V(t)}\qquad(0<s\le t).
\end{equation}
The upper volume bound and density rigidity hold in the noncollapsed
setting; see \cite{CC1997I,DePhilippisGigli2018}.

\begin{lemma}[Local lower volume bound]\label{lem:lower}
	For every $L>0$ there is $c_L>0$ such that
	\begin{equation}\label{eq:lower}
		\mu(B_t(z))\ge c_Lt^n
		\qquad (z\in B_L(p),\ 0<t\le1).
	\end{equation}
	There is also a constant $b_n<\infty$ such that
	\begin{equation}\label{eq:upper}
		\mu(B_t(z))\le V(t)\le b_nt^n\qquad(0<t\le1).
	\end{equation}
\end{lemma}
\begin{proof}
	Volume convergence gives $\mu(B_1(p))\ge v$. If $z\in B_L(p)$, then
	$B_1(p)\subset B_{L+1}(z)$. By \eqref{eq:BG},
	\[
	\mu(B_t(z))\ge
	\mu(B_{L+1}(z))\frac{V(t)}{V(L+1)}
	\ge\frac{v\omega_n}{V(L+1)}t^n.
	\]
	Here $t\le1<L+1$ and $V(t)\ge\omega_nt^n$. This proves
	\eqref{eq:lower}; the upper bound follows from volume comparison and the
	continuous extension of $V(t)/t^n$ to $[0,1]$.
\end{proof}

\begin{lemma}[Density and the regular set]\label{lem:density}
	The density in \eqref{eq:density} exists at every point and satisfies
	\begin{equation}\label{eq:theta-properties}
		0<\Theta(x)\le1,\qquad
		\Theta(x)=\lim_{r\rightarrow0}\frac{\mu(B_r(x))}{\omega_nr^n},
		\qquad \Reg=\{x:\Theta(x)=1\}.
	\end{equation}
	Furthermore, $\Theta$ is lower semicontinuous, each $S_\delta$ is
	closed, and
	\begin{equation}\label{eq:sing-union}
		\Sing=\bigcup_{m=2}^{\infty}S_{1/m}.
	\end{equation}
	For $x\in S_\delta$ one has
	\begin{equation}\label{eq:persistent}
		\mu(B_t(x))\le(1-\delta)V(t)\qquad(t>0).
	\end{equation}
\end{lemma}
\begin{proof}
	The limit exists by \eqref{eq:BG}. Its positivity follows from
	Lemma~\ref{lem:lower}, and its upper bound follows from noncollapsed
	volume comparison. The two expressions for the density agree because
	$V(r)/(\omega_nr^n)\to1$.
	
	We recall the relation between volume density and metric regularity.
	Along any blow-up sequence converging to a tangent cone, volume
	convergence shows that every ball of radius $s$ centered at the vertex
	has volume
	$\Theta(x)\omega_ns^n$. If $\Theta(x)=1$, maximal-volume rigidity
	identifies this tangent cone with $\R^n$. The argument applies to every
	tangent cone. Conversely, if $x$ is regular, volume convergence to a
	Euclidean tangent gives $\Theta(x)=1$. These are standard
	consequences of noncollapsed volume rigidity; see \cite{CC1996,CC1997I,DePhilippisGigli2018}.
	
	For completeness, we verify that, for each fixed radius, the ball volume
	is a continuous function of the center. First, \eqref{eq:BG} gives
	\[
	\mu(B_{t+h}(x))\le\frac{V(t+h)}{V(t)}\mu(B_t(x)).
	\]
	Letting $h\rightarrow0$ proves that every metric sphere has zero
	$\mu$-measure. If $x_j\to x$, the inclusions
	$B_{t-h}(x)\subset B_t(x_j)\subset B_{t+h}(x)$ for large $j$, followed
	by $h\rightarrow0$, prove continuity. Thus
	\[
	\Theta(x)=\sup_{k\ge1}\frac{\mu(B_{1/k}(x))}{V(1/k)}
	\]
	is lower semicontinuous. Hence $S_\delta$ is closed.
	Equation~\eqref{eq:sing-union} follows from
	\eqref{eq:theta-properties}, and \eqref{eq:persistent} follows from
	\eqref{eq:BG} by letting the smaller radius tend to zero.
\end{proof}

\begin{lemma}[Propagation of density deficits to nearby centers]\label{lem:propagation}
	For every $0<\delta<1$ there is $\eta=\eta(n,\delta)\in(0,1/4)$
	such that, for $x\in S_\delta$ and $0<r\le1$,
	\begin{equation}\label{eq:propagation}
		D_r(y)\ge\frac\delta2\qquad(y\in B_{\eta r}(x)).
	\end{equation}
\end{lemma}
\begin{proof}
	The ratio $V((1+\eta)r)/V(r)$ extends continuously to $r=0$,
	with value $(1+\eta)^n$, and converges uniformly to $1$ on $[0,1]$
	as $\eta\rightarrow0$. We may therefore choose $\eta\in(0,1/4)$
	sufficiently small that
	\begin{equation}\label{eq:model-ratio}
		V((1+\eta)r)\le(1+\delta/2)V(r)\qquad(0<r\le1).
	\end{equation}
	For $y\in B_{\eta r}(x)$,
	$B_r(y)\subset B_{(1+\eta)r}(x)$, so
	\[
	\frac{\mu(B_r(y))}{V(r)}
	\le(1-\delta)\frac{V((1+\eta)r)}{V(r)}
	\le(1-\delta)(1+\delta/2)\le1-\delta/2.
	\]
	This proves \eqref{eq:propagation}.
\end{proof}

\section{Proof of the main theorem}\label{sec:proof}

Fix $R>0$ and $0<\delta<1$, and set $F=S_\delta\cap B_R(p)$.
Let $\eta$ be as in Lemma~\ref{lem:propagation}. All radii considered
below are chosen sufficiently small for \textup{(A)} to apply on
$B_{R+1}(p)$.

\subsection{Packing and Hausdorff measure}
For $0<r<\min\{1,r_{R+1}\}$, choose a maximal family of pairwise disjoint balls
\[
\{B_{\eta r}(x_j)\}_{j=1}^{N},\qquad x_j\in F.
\]
This family is finite: its balls lie in $B_{R+1}(p)$, their volumes
are bounded below by a uniform positive constant by Lemma~\ref{lem:lower},
and their total volume is at most $\mu(B_{R+1}(p))$. By maximality,
\begin{equation}\label{eq:cover}
	F\subset\bigcup_{j=1}^{N}B_{2\eta r}(x_j).
\end{equation}
Indeed, if a point of $F$ were not covered, its ball of radius $\eta r$
could be added to the disjoint family.

By \textup{(A)} and Lemmas~\ref{lem:propagation} and~\ref{lem:lower},
\begin{align}
	C_{R+1}r^3
	&\ge\int_{B_{R+1}(p)}D_r(y)^{3/2}\,d\mu(y)\notag\\
	&\ge\sum_{j=1}^{N}\int_{B_{\eta r}(x_j)}D_r(y)^{3/2}\,d\mu(y)
	\notag\\
	&\ge N(\delta/2)^{3/2}c_{R+1}\eta^nr^n.
	\label{eq:packing-calculation}
\end{align}
Consequently,
\begin{equation}\label{eq:packing}
	N\le A_{R,\delta}r^{3-n},\qquad
	A_{R,\delta}=
	\frac{C_{R+1}}{(\delta/2)^{3/2}c_{R+1}\eta^n}.
\end{equation}
Taking $\rho=2\eta r$ in \eqref{eq:cover}--\eqref{eq:packing}
yields \eqref{eq:main-cover}, after adjusting the constant.

The same covering satisfies
\[
\sum_{j=1}^{N}\bigl(\diam B_{2\eta r}(x_j)\bigr)^{n-3}
\le A_{R,\delta}(4\eta)^{n-3}.
\]
Letting $r\rightarrow0$ proves \eqref{eq:main-finite}, with the fixed
normalization factor in Hausdorff measure absorbed into the constant. For every $s>n-3$,
\[
\sum_{j=1}^{N}\bigl(\diam B_{2\eta r}(x_j)\bigr)^s
\le A_{R,\delta}(4\eta)^sr^{s+3-n}\longrightarrow0.
\]
Thus $\dimH F\le n-3$. Taking the union over positive integers $R$
and using \eqref{eq:sing-union}, we obtain \eqref{eq:main-dim}
and the sigma-finiteness of $\HH^{n-3}$ on $\Sing$.

\subsection{Tubular volume}
We prove \eqref{eq:main-tube} directly. Set
$r=\rho/\eta$. If $\rho$ is sufficiently small, then
\[
\Tube_\rho(F)\subset B_{R+1}(p),\qquad 0<r<\min\{1,r_{R+1}\}.
\]
Each point of $\Tube_\rho(F)$ lies in a ball $B_{\eta r}(x)$ with
$x\in S_\delta$. By Lemma~\ref{lem:propagation},
\begin{align*}
	(\delta/2)^{3/2}\mu(\Tube_\rho(F))
	&\le\int_{\Tube_\rho(F)}D_{\rho/\eta}(y)^{3/2}\,d\mu(y)\\
	&\le C_{R+1}(\rho/\eta)^3.
\end{align*}
This gives the asserted bound.

\subsection{The closed topological exceptional set}
We use the following standard consequence of the metric Reifenberg
theorem of Cheeger--Colding \cite[Appendix~1]{CC1997I}:
\begin{equation}\label{eq:topological-input}
	\text{every }x\in\Reg\text{ has an open neighborhood in }X
	\text{ homeomorphic to an open subset of }\R^n.
\end{equation}
This local formulation is recalled in the introduction of
\cite{BruePigatiSemola2024+}. The neighborhoods in
\eqref{eq:topological-input} are open in $X$, rather than merely
in the relative topology of $\Reg$.

The locus $\Man$ defined in \eqref{eq:manifold-locus} is open: every
point in a Euclidean chart has a smaller neighborhood homeomorphic to
$\R^n$. It contains $\Reg$ by \eqref{eq:topological-input}. Therefore
\[
E=X\setminus\Man\text{ is closed},\qquad E\subset\Sing.
\]
The dimension bound and sigma-finiteness follow from the corresponding
properties of $\Sing$. Moreover, $\Man$ is Hausdorff and second
countable, being an open subset of the separable metric space $X$.
Together with the local charts, these properties show that $\Man$ is
a topological $n$-manifold without boundary. This completes the proof of
Theorem~\ref{thm:main}.\qed

\section{Comparison with Uniform almost maximal volume}\label{sec:examples}

Fix $0<\varepsilon<1$. We consider the following uniform almost
maximal volume condition on bounded sets: for every $R>0$, there is
$s_R>0$ such that
\begin{equation}\tag{AMV$_\varepsilon$}\label{eq:AMV}
	D_r(x)\le\varepsilon
	\qquad(x\in B_R(p),\ 0<r<s_R).
\end{equation}
Here $\varepsilon$ is independent of $r$, while the admissible radius
may depend on $R$. By \eqref{eq:BG}, a bound at a fixed radius implies
the same bound at all smaller radii with the same centers. The following
example shows that this condition does not imply \textup{(A)}.

\begin{example}[Almost maximal volume without \textup{(A)}]\label{ex:cone}
	Let $0<\alpha<1$, let $C_\alpha$ be the two-dimensional Euclidean cone
	of total angle $2\pi\alpha$, and set
	\[
	X_\alpha=\R^{n-2}\times C_\alpha,
	\qquad S=\R^{n-2}\times\{o\}.
	\]
	The metric on $C_\alpha\setminus\{o\}$ is
	$ds^2+\alpha^2s^2d\theta^2$, with $0\le\theta<2\pi$.
	Smoothing the cone tip yields noncollapsed smooth approximations with
	nonnegative sectional curvature; see \cite[Example~2.14]{CheegerNaber2015}.
	
	The coordinate map $F:\R^n\to X_\alpha$, which preserves the
	Euclidean and polar coordinates, is a bijective $1$-Lipschitz map:
	the target metric contracts angular lengths by the factor $\alpha$.
	Away from the axis, its volume Jacobian is $\alpha$.
	For any $x=F(\widetilde x)$,
	$F(B_r(\widetilde x))\subset B_r(x)$, and hence
	\begin{equation}\label{eq:cone-lower}
		\mu(B_r(x))\ge\alpha\omega_nr^n.
	\end{equation}
	Therefore
	\[
	D_r(x)\le1-\alpha\frac{\omega_nr^n}{V(r)}
	\le(1-\alpha)+O_n(r^2).
	\]
	Given $0<\varepsilon<1$, choose $\alpha=1-\varepsilon/2$ and then
	choose $s_\varepsilon>0$ sufficiently small. It follows that $X_\alpha$ satisfies
	\eqref{eq:AMV} at all centers for $0<r<s_\varepsilon$.
	
	On the other hand, $\Theta(z)=\alpha$ for every $z\in S$. Set
	$\delta=1-\alpha$ and choose $\eta>0$ such that
	$\alpha(1+\eta)^n\le1-\delta/2$. If $\dist(y,S)<\eta r$ and $z$
	is its nearest point on $S$, then
	\[
	\mu(B_r(y))\le\mu(B_{(1+\eta)r}(z))
	=\alpha\omega_n(1+\eta)^nr^n
	\le(1-\delta/2)V(r).
	\]
	Thus $D_r(y)\ge\delta/2$ throughout this tube. Write $p=(u_0,o)\in S$.
	For all sufficiently small $r$, the product
	$B_{R/2}^{\R^{n-2}}(u_0)\times B_{\eta r}^{C_\alpha}(o)$ is contained
	in $B_R(p)$ and has volume
	$\omega_{n-2}(R/2)^{n-2}\pi\alpha\eta^2r^2$. It follows that
	\[
	\int_{B_R(p)}D_r^{3/2}\,d\mu
	\ge c_{n,R,\alpha}r^2.
	\]
	Since $r^2$ is not $O(r^3)$ as $r\rightarrow0$, condition \textup{(A)}
	fails. The space $X_\alpha$ is homeomorphic to $\R^n$, so
	\textup{(A)} is not necessary for topological manifold regularity.
\end{example}

\bibliographystyle{plain-no-oxford}
\bibliography{reference}

@article{Naber2020,
	title={Conjectures and open questions on the structure and regularity of
	spaces with lower Ricci curvature bounds},
	author={Naber, A. },
	journal={Symmetry, Integrability and Geometry: Methods and Applications },
	volume={103},
	number={16},
	pages={1-8},
	year={2020},
}

@book{Cheeger2001,
	title={Degeneration of Riemannian metrics under Ricci curvature bounds},
	author={Cheeger, J.},
	note={Classe di Science},
	publisher ={Accademia nazionale dei Lincei},
	address = {Scuola Normale superiore},
	year={2001},
}

@Article{CheegerColdingTian2002,
	author={J. Cheeger and T,H. Colding and G. Tian},
	title={On the singularities of spaces with bounded Ricci cur-
	vature},
	journal={Geom. Funct. Anal.},
	year="2002",
	volume="12",
	number={5},
	pages="873?914",
}

@Article{DePhilippisGigli2018,
	author={G. De Philippis and N. Gigli},
	title={Non-collapsed spaces with Ricci curvature bounded from below},
	journal={J. \`Ec. polytech. Math.},
	year="2018",
	volume=" ",
	number={5},
	pages="613--650",
}

@unpublished{BruePigatiSemola2024+,
	author       = {E. Bru\`e and A. Pigati and D.Semola},
	title        = {Topological regularity and stability of noncollapsed spaces with
	Ricci curvature bounded below},
	note         = {arXiv:2405.03839},
	month        = Mar,
	year         = 2024
}

@Article{CheegerNaber2015,
	author="J. Cheeger and A. Naber",
	title="Regularity of Einstein manifolds and the codimension 44 conjecture",
	journal="Ann. Math.",
	year="2015",
	volume="182",
	pages="1093--1165",
	doi="10.2140/gt.2016.20.2575",
	issue="3",
}

@article{CC1996,
	title={Lower Bounds on {R}icci Curvature and the Almost Rigidity of Warped Products},
	author={Cheeger, J. and Colding, T. H.},
	journal={Ann. of Math.},
	volume={144},
	number={1},
	pages={189-237},
	year={1996},
}

@article{Colding1997,
	title={Ricci Curvature and Volume Convergence},
author={Colding, T. H.},
journal={Ann. of Math.},
volume={145},
number={3},
pages={477--501},
year={1997},
}

@article{CC1997I,
	title={On the Structure Of Spaces with {R}icci Curvature Bounded Below. {I}},
author={Cheeger, J. and Colding, T. H.},
journal={J. Differential Geom.},
volume={46},
number={3},
pages={406--480},
year={1997},
}
\end{document}